\documentclass[11pt,a4paper,twoside]{article}
\usepackage{graphics}
\usepackage{amssymb}
\usepackage{amsmath}
\usepackage{mathrsfs}
\usepackage{makeidx}
\usepackage{color}
\usepackage{graphicx}
\usepackage{subfigure}
\usepackage{multicol}
\usepackage{multirow}
\usepackage{float}
\usepackage{booktabs}
\usepackage{epsfig}
\usepackage{multirow}
\usepackage{epstopdf}
\usepackage[colorlinks,
            linkcolor=red,
            anchorcolor=blue,
            citecolor=blue
            ]{hyperref} %引用超链接宏包
\usepackage{caption}
\usepackage{algorithm,algorithmic}

\usepackage{latexsym, cite, bm, amsthm, amsmath}
\usepackage{enumerate}
\usepackage{marginnote}
\usepackage{epstopdf}

\def \[{\begin{equation}}
\def \]{\end{equation}}

\newtheorem{thm}{Theorem}[section]

\newtheorem{lem}{Lemma}[section]

\newtheorem{defn}{Definition}[section]

\newtheorem{rem}{Remark}[section]

\newtheorem{exam}{Example}[section]
\numberwithin{equation}{section}

\def\fix{\pmb{Fix}}

\title{\bf A fixed-time inverse-free dynamical system for solving the system of absolute value equations}
\usepackage[marginal]{footmisc}
\usepackage{authblk}
\author[a]{Xuehua Li\thanks{Email address: 3222714384@qq.com.}}
\author[b]{Dongmei Yu\thanks{Supported partially by the Ministry of Education in China of Humanities and Social Science Project (Grand No. 21YJCZH204), the Natural Science Foundation of Liaoning Province (Grand Nos. 2020-MS-301, LJ2020ZD002) and the Youth Talent Entrustment Project of Liaoning Provincial Federation Social Science Circles (Grand No. 2022lslwtkt-069). Email: {yudongmei1113@163.com}.}}
\author[c]{Yinong Yang\thanks{Supported partially by the National Natural Science Foundation of China (Grant
No. 12101281). Email address: yynmath@163.com.}}
\author[d]{Deren Han\thanks{Supported partially by the National Natural Science Foundation of China (Grant Nos. 12131004 and 11625105). Email address: {handr@buaa.edu.cn}.}}
\author[a]{Cairong Chen\thanks{Corresponding author. Supported partially by the National Natural Science Foundation of China (Grant No. 11901024) and  the Natural Science Foundation of Fujian Province (Grand No. 2021J01661). Email address: cairongchen@fjnu.edu.cn.}}
\affil[a]{School of Mathematics and Statistics, FJKLMAA and Center for Applied Mathematics of Fujian Province, Fujian Normal University, Fuzhou, 350007, P.R. China}
\affil[b]{Institute for Optimization and Decision Analytics, Liaoning Technical University, Fuxin, 123000, P.R. China}
\affil[c]{School of Mathematics, Liaoning University, Shenyang, 110036, P.R. China}
\affil[d]{LMIB of the Ministry of Education, School of Mathematical Sciences, Beihang University, Beijing 100191, P.R. China}

\begin{document}
\date{}
\maketitle

\begin{quote}
{\bf Abstract:} In this paper, an inverse-free dynamical system with fixed-time convergence is presented to solve the system of absolute value equations (AVEs). Under a mild condition, it is proved that the solution of the proposed dynamical system converges to the solution of the AVEs. Moreover, in contrast to the existing inverse-free dynamical system \cite{chen2021}, a conservative settling-time of the proposed method is given. Numerical simulations illustrate the effectiveness of the new method.

{\small
%\medskip
%{\em Mathematics Subject Classification}. 90C33, 65K05

}
\end{quote}

\section{Introduction}\label{sec:Introduction}
To solve the system of absolute value equations (AVEs) is to find an $x \in \mathbb{R}^n$ such that
\begin{equation}\label{eq:ave}
  Ax - |x| - b = 0,
\end{equation}
where $A \in \mathbb{R}^{n\times n}$, $b\in \mathbb{R}^{n}$ and $|x|$ represents the
componentwise absolute value of the unknown vector $x$. The AVEs \eqref{eq:ave} is a special case of the generalized absolute value equations~(GAVEs)
\begin{equation}\label{eq:gave}
	Cx - D|x| - c = 0,
\end{equation}
where $C,\, D\in \mathbb{R}^{m\times n}$, $x\in \mathbb{R}^{n}$ and $c\in \mathbb{R}^{m}$. The  GAVEs~\eqref{eq:gave} is originally introduced by Rohn in \cite{rohn2004} and further investigated in \cite{mang2007,prok2009,hlad2018} and the references therein. Over the past two decades, the AVEs~\eqref{eq:ave} has been widely studied in the optimization community because of its relevance to many mathematical programming problems, such as the linear complementarity problem~(LCP) \cite{prok2009,mang2007,mame2006,hu2010}, the horizontal LCP~(HLCP) \cite{mezzadri2020}, the generalized LCP (GLCP)~\cite{mame2006} and others, see e.g. \cite{prok2009,mame2006,mang2009} and the references therein. In addition, the AVEs~\eqref{eq:ave} is closely related to the system of linear interval equations \cite{rohn1989}.

In general, it has been shown in \cite{mang2007} that solving the GAVEs~\eqref{eq:gave} is NP-hard.  Moreover, if the GAVEs~\eqref{eq:gave} is solvable, it follows from \cite{prok2009} that checking whether the GAVEs~\eqref{eq:gave} has a unique solution or multiple solutions is NP-complete. Throughout this paper, we assume that $A$ is invertible and $\|A^{-1}\| <1$, and thus the AVEs~\eqref{eq:ave} has a unique solution for any $b\in \mathbb{R}^n$ \cite{mame2006}. For more discussions about the unique solvability of the AVEs~\eqref{eq:ave}, see e.g. \cite{wuli2018,mezzadri2020,rohn2014,wugu2016,hladik2022} and the references therein.

In this paper, we focus on further studying the continuous solution schemes for solving the AVEs~\eqref{eq:ave}. We briefly introduce some existing works in the following. To this end, we recall two reformulations of the AVEs~\eqref{eq:ave}.

The AVEs~\eqref{eq:ave} is equivalent with the following GLCP \cite{mame2006}:
\begin{equation}\label{eq:glcp}
    Q(x)\doteq Ax+x-b\ge 0,\quad  F(x)\doteq Ax-x-b\ge 0 \quad \hbox{and}\quad \left\langle  Q(x), F(x)\right\rangle=0.
\end{equation}
Furthermore, if $1$ is not an eigenvalue of $A$, then the AVEs~\eqref{eq:ave} can be reformulated as the following LCP \cite{mame2006}:
\[\label{lcps}
u\ge 0,\quad (A+I)(A-I)^{-1}u+q\ge0 \quad \hbox{and}\quad \left\langle u,(A+I)(A-I)^{-1}u+q\right\rangle=0
\]
with
\begin{equation}\label{eq:ux}
q=\left[(A+I)(A-I)^{-1}-I\right]b,\quad u=(A-I)x-b.
\end{equation}
It follows from \eqref{lcps} and \eqref{eq:ux} that if $u^*$ is a solution of the LCP~\eqref{lcps}, then $x^* = (A-I)^{-1}(u^*+b)$ is a solution of the AVEs~\eqref{eq:ave}.

By utilizing the reformulation \eqref{lcps} with \eqref{eq:ux} of the AVEs~\eqref{eq:ave}, some dynamical systems are constructed to solve the AVEs~\eqref{eq:ave}. For instance, the following dynamical model
\begin{align*}%\label{eq:maee}
&\text{state~equation:}\quad \frac{du}{dt} =P_{\Omega} [u - \lambda g(u,\beta)] - u,\\\label{eq:maee2}
&\text{output~equation:} \quad x = (A-I)^{-1}(u+b)
\end{align*}
is used by Mansoori, Eshaghnezhad and Effati \cite{maee2017} to solve the AVEs~\eqref{eq:ave}, where $0<\lambda \le 1$, $0<\beta< \frac{1}{5L}$, and
\begin{align*}
&g(u,\beta) = e(u,\beta) - \beta Me(u,\beta),\\
&e(u,\beta) = u - P_{\Omega} [u - \beta(Mu + q)],\\
&L = \|M\|,\; M = (A+I)(A-I)^{-1},\\
&\Omega=\{x\in\mathbb{R}^n|x\ge 0\}.
\end{align*}
Huang and Cui \cite{huang2017} use the following dynamical system
\begin{align*}%\label{eq:hucu1}
&\text{state~equation:}\quad \frac{du}{dt} = -\gamma e(u,1),\\%\label{eq:hucu2}
&\text{output~equation:} \quad x = (A-I)^{-1}(u+b)
\end{align*}
to solve the AVEs~\eqref{eq:ave} and Mansoori and Erfanian \cite{maer2018} suggest the following dynamical system
\begin{align*}%\label{eq:maer1}
&\text{state~equation:}\quad \frac{du}{dt} = -\gamma(I+M^\top) e(u,1),\\%\label{eq:maer2}
&\text{output~equation:} \quad x = (A-I)^{-1}(u+b)
\end{align*}
to solve the AVEs~\eqref{eq:ave}, where $\gamma>0$ is the convergence rate. Recently, Ju, Li, Han and He \cite{ju2022} develop a fixed-time dynamical system for solving the AVEs~\eqref{eq:ave}. Based on the relation with HLCP, the following dynamical system is proposed in \cite{gao2014} to solve AVEs~\eqref{eq:ave}:
\begin{align*}%\label{eq:gao1}
&\text{state~equation:}\quad \frac{dz}{dt} = \frac{\rho}{2}(|x|-z),\\%\label{eq:gao2}
&\text{output~equation:} \quad x = A^{-1}(z+b),
\end{align*}
where $\rho >0$ is a scaling constant. Obviously, all of the above mentioned dynamical systems involve the inversion of the matrix $A-I$ or $A$. In order to avoid the inversion, based on \eqref{eq:glcp}, Chen, Yang, Yu and Han \cite{chen2021} develop the following inverse-free dynamical system
\begin{equation}\label{eq:dymos}
\frac{dx}{dt} = \gamma A^\top(b + |x| - Ax)\doteq -\gamma A^\top r(x)\doteq g(\gamma, x),
\end{equation}
where $\gamma>0$ is the convergence rate parameter. The dynamical system \eqref{eq:dymos} can trace back to \cite{xia2004,liu2010,hu2007,gao2001}. Subsequently, an inertial version of \eqref{eq:dymos} is developed in \cite{yu2022}. Other inverse-free dynamical system appears in \cite{saheya2019,wang2015,yong2020}, in which the smoothing technique is used.

It is known that the finite-time convergence proposed in \cite{bhat2000} is of practical interests than the classic asymptotical stability or exponential stability over infinite time \cite{ju2022}. To overcome the limitation that the settling-time of the finite-time model is initial condition dependent, the concept of the so-called fixed-time convergence \cite{polyakov2011} is developed. Inspired by the works of \cite{chen2021,ju2022}, we will develop a fixed-time inverse-free dynamical system for solving the AVEs~\eqref{eq:ave}. The main features of our method can be summarized as follows.

\begin{itemize}
\item[(a)] Comparing with the methods proposed in \cite{maee2017,huang2017,maer2018,gao2014}, the proposed method is inverse-free and possesses a conservative settling-time.

\item[(b)] Comparing with the method proposed in \cite{ju2022}, the proposed method is inverse-free.

\item[(c)] Comparing with the methods proposed in \cite{chen2021,yu2022}, the proposed method has a conservative settling-time.

\item[(d)] Comparing with the methods proposed in \cite{saheya2019,wang2015,yong2020}, the proposed method has a conservative settling-time and does not need any smoothing function.

\end{itemize}

The rest of this paper is organized as follows. In section \ref{sec:pre} we state a few basic results on the AVEs~\eqref{eq:ave} and the autonomous system, relevant to our later developments. The fixed-time inverse-free dynamical system to solve the AVEs~\eqref{eq:ave} is developed in section \ref{sec:main} and its convergence analysis is also given there.
Numerical simulations are given in section \ref{sec:numerical}. Conclusions are made in section \ref{sec:conclusion}.

\textbf{Notation.} We use $\mathbb{R}^{n\times n}$ to denote the set of all $n \times n$ real matrices and $\mathbb{R}^{n}= \mathbb{R}^{n\times 1}$. We use $\mathbb{R}_+$ to denote the nonnegative reals. $I$ is the identity matrix with suitable dimension. $| \cdot |$ denotes absolute value for real scalar. The transposition of a matrix or vector is denoted by $\cdot ^\top$. The inner product of two vectors in $\mathbb{R}^n$ is defined as $\langle x, y\rangle\doteq x^\top y= \sum\limits_{i=1}^n x_i y_i$ and $\| x \|\doteq\sqrt{\langle x, x\rangle} $ denotes the $2$-norm of vector $x\in \mathbb{R}^{n}$. $\|A\|$ denotes the spectral norm of $A$ and is defined by the formula $\| A \|\doteq \max \left\{ \| A x \| : x \in \mathbb{R}^{n}, \|x\|=1 \right\}$. The smallest singular value and the smallest eigenvalue of~$A$ are denoted by $\sigma_{\min}(A)$ and $\lambda_{\min}(A)$, respectively. $\textbf{tridiag}(a, b, c)$ denotes a matrix that has $a, b, c$ as the subdiagonal, main diagonal and superdiagonal entries in the matrix, respectively. The projection mapping from $\mathbb{R}^n$ onto $\Omega$, denoted by $P_{\Omega}$, is defined as $P_{\Omega}[x]=\arg\min\{\|x-y\|:y\in \Omega\}$.

\section{Preliminaries}\label{sec:pre}
In this section, we collect a few important results on the autonomous system and the AVEs~\eqref{eq:ave}, which lay the foundation of our later arguments.

Before talking something about the autonomous system, we give the definition of Lipschitz continuity.
\begin{defn}
The function $F:\mathbb{R}^n\rightarrow \mathbb{R}^n$ is said to be Lipschitz continuous with Lipschitz constant $L>0$ if
$$
\|F(x)-F(y)\|\le L\|x-y\|,\quad \forall x,y \in \mathbb{R}^n.
$$
\end{defn}

Consider the autonomous system
\begin{equation}\label{eq:auto-dysm}
\frac{dx}{dt} = f(x(t)),\quad x(0) = x_0\in \mathbb{R}^n,
\end{equation}
where $f$ is a function from $\mathbb{R}^n$ to $\mathbb{R}^n$. Throughout this paper, denote $x(t;x_0)$ the solution of~\eqref{eq:auto-dysm} determined by the initial value condition $x(0) = x_0$.

\begin{lem}(\cite{khalil1996})\label{lem:solution}
Assume that $f:\mathbb{R}^n\rightarrow\mathbb{R}^n$ is a continuous function, then for arbitrary $x(0)=x_0\in \mathbb{R}^n$, there exists a local solution $x(t;x_0),\,t\in[0,\tau]$ for some $\tau> 0$. Furthermore, if $f$ is locally Lipschitz continuous at $x_0$, then the solution is unique; and if $f$ is Lipschitz continuous in $\mathbb{R}^n$, then $\tau$ can be extended to $+\infty$.
\end{lem}

\begin{defn}(\cite{khalil1996}) Let $x^*\in \mathbb{R}^n$, then it is called an equilibrium point of the dynamical system \eqref{eq:auto-dysm} if $f(x^*) = 0$.
\end{defn}

\begin{lem}(\cite{khalil1996})\label{thm:als}
Let $x^*$ be an equilibrium point for \eqref{eq:auto-dysm}. Let $V:\mathbb{R}^n \rightarrow \mathbb{R}$ be a continuously differentiable function such that
$$
V(x^*) = 0 \quad \text{and} \quad V(x)> 0,\;\forall x\neq  x^*,
$$
$$\frac{dV(x)}{dt}< 0,\;\forall x\neq  x^*,
$$
$$
\|x-x^*\|\rightarrow \infty \Rightarrow V(x)\rightarrow\infty,
$$
then $x^*$ is globally asymptotically stable.
\end{lem}

\begin{defn}(\cite{polyakov2011})
The equilibrium point $x^*\in \mathbb{R}^n$ of the dynamical system \eqref{eq:auto-dysm} is said to be globally finite-time stable if it is globally asymptotically stable and any solution $x(t;x_0)$ of \eqref{eq:auto-dysm} reaches the equilibria at some finite time moment, i.e., $x(t;x_0) = x^*, \forall t\ge T(x_0)$, where $T:\mathbb{R}^n\rightarrow \mathbb{R}_+$ is the so-called settling-time function.
\end{defn}

\begin{defn}(\cite{polyakov2011})
The equilibrium point $x^*\in \mathbb{R}^n$ of the dynamical system \eqref{eq:auto-dysm} is said to be globally fixed-time stable if it is globally finite-time stable and the settling-time function $T(x_0)$ is bounded by some positive number $T_{\max}>0$, i.e., $T(x_0) \le T_{\max}$ for any $x_0\in \mathbb{R}^n$.
\end{defn}

\begin{lem}(\cite{polyakov2011})\label{lem:ft}
Let $x^*\in \mathbb{R}^n$ be the equilibrium point of the dynamical system \eqref{eq:auto-dysm}. If there exists a continuous radially unbounded function $V:\mathbb{R}^n\rightarrow \mathbb{R}_+$ such that
\begin{itemize}
  \item [(a)] $V(x) = 0 \quad \Leftrightarrow x = x^*;$

  \item [(b)] any solution $x(t;x_0)$ of the dynamical system \eqref{eq:auto-dysm} satisfies the following inequality
$$
\frac{dVx(t;x_0)}{dt} \le -\alpha V(x(t;x_0))^{k_1} - \beta V(x(t;x_0))^{k_2}
$$
for some $\alpha>0,\beta>0$, $0<k_1<1,\,k_2>1$.
\end{itemize}
Then the equilibrium point $x^*$ of \eqref{eq:auto-dysm} is globally fixed-time stable with
settling-time
$$
T(x_0) \le T_{\max} = \frac{1}{\alpha(1-k_1)} + \frac{1}{\beta(k_2 - 1)},\forall x_0\in \mathbb{R}^n.
$$
\end{lem}

Before ending this section, we will give some properties of the AVEs~\eqref{eq:ave}.

\begin{lem}(\cite{chen2021})\label{thm:contrac}
If $\sigma_{\min}(A)> 1$ and $x^*$ is the solution of the AVEs \eqref{eq:ave}, then
\begin{equation}\label{ie:contr}
(x-x^*)^\top A^\top r(x) \ge\frac{1}{2} \left\|r(x)\right\|^2,\; \forall \,x\in \mathbb{R}^n.
\end{equation}
\end{lem}

\begin{lem}(\cite{chen2021})\label{lem:errbound}
Assume $\sigma_{\min}(A)>1$, then the AVEs~\eqref{eq:ave} has a unique solution, say $x^*$, and
\begin{equation}\label{ie:lowerbound}
\frac{1}{L_1 + L_2}\|r(x)\|\le \|x-x^*\|\le \frac{L_1 + L_2}{\mu} \|r(x)\|, \quad \forall x\in \mathbb{R}^n,
\end{equation}
where $L_1=\|A+I\|$ and $L_2=\|A-I\|$ are Lipschitz constants of the functions $Q$ and $F$ defined as in \eqref{eq:glcp}, respectively, and $0 < \mu = \sigma_{\min}(A)^2 - 1$.
\end{lem}

\section{The new dynamical system and its convergence analysis}\label{sec:main}
In this section, inspired by the works of \cite{chen2021,ju2022}, we present a fixed-time inverse-free dynamical system for solving the AVEs~\eqref{eq:ave}.

The developed fixed-time inverse-free dynamical system is as follows:
\begin{equation}\label{eq:nds}
\begin{array}{l}
\frac{dx}{dt}=\rho(x) g(\gamma,x),\\
\rho(x)= \begin{cases}
		\frac{\rho_{1}}{\Vert r(x)\Vert^{1-\lambda_{1}}}+\frac{\rho_{2}}{\Vert r(x)\Vert^{1-\lambda_{2}}}, &\text{if}\quad x \notin \fix(r),\\
		0,&\text{if}\quad x\in \fix(r),
	\end{cases}
\end{array}
\end{equation}
where $\rho_{1},\,\rho_{2},\gamma>0$, $\lambda_1\in (0,1)$, $\lambda_2\in (1,+\infty)$ and $\fix(r) = \{x\in \mathbb{R}^n: r(x) = 0\}$.

\begin{thm}\label{thm:equi} Let $A$ be a nonsingular matrix, then $x^{*}$ is an equilibrium of \eqref{eq:nds} if and only if it solves the AVEs~\eqref{eq:ave}.
\end{thm}
\begin{proof} If $x^{*}$ is an equilibrium point of \eqref{eq:nds}, then
\begin{equation*}
\rho(x^{*})A^{T}r(x^*)=0.
\end{equation*}
Since $A$ is invertible, the above equation implies that
\begin{equation*}
\rho(x^{*})=0 \quad \text{or} \quad \Vert r(x^{*})\Vert=0,
\end{equation*}
from which we have
\begin{equation*}
x^*\in \fix(r)\quad \text{or}\quad  r(x^{*}) = 0,
\end{equation*}
both of which mean that $x^*$ is a solution of the AVEs~\eqref{eq:ave}.

The other direction is trivial.
\end{proof}

\begin{lem}(\cite{chen2021})\label{lem:lipcon}
For any given $\gamma>0$, the function $g$ defined as in \eqref{eq:dymos} is Lipschitz continuous with respect to $x$ in $\mathbb{R}^n$  with Lipschitz constant $\gamma \|A^\top\|(\|A\|+1)$.
\end{lem}

The following theorem follows from Lemma~\ref{lem:solution} and Lemma~\ref{lem:lipcon}.

\begin{thm}
For a given initial value $x(0)=x_0$, there exists a unique solution $x(t;x_0),t\in [0,\infty)$ for the dynamical system \eqref{eq:nds}.
\end{thm}

Now we are in the position to explore the convergence of the proposed model \eqref{eq:nds}.

\begin{thm}\label{thm:convergence}
Let $A$ be a nonsingular matrix and $\|A^{-1}\|<1$. Assume that $x^*$ is the unique solution of the AVEs~\eqref{eq:ave}. Then $x^*$ is the globally fixed-time stable equilibrium point of \eqref{eq:nds} with the settling-time given as
\begin{equation}\label{ie:st}
T(x_0)\leq T_{\max}=\frac{1}{c_{1}(1-\kappa_{1})}+\frac{1}{c_{2}(\kappa_{2}-1)},
\end{equation}
where $x_0$ is the initial condition, $c_{1}>0$, $c_{2}> 0$, $ \kappa_{1} \in(0.5,1)$ and $\kappa_{2}\in(1 +\infty)$ are some constants defined as in \eqref{eq:cc} and \eqref{eq:kk}.
\end{thm}
\begin{proof}
It follows from Theorem~\ref{thm:equi} that $x^*$ is the equilibrium point of \eqref{eq:nds}. Consider the following Lyapunov function
\begin{equation*}
V(x)=\frac{1}{2} \Vert x-x^{*}\Vert^{2}.
\end{equation*}
Obviously, $V(x)\rightarrow \infty$ as $\Vert x-x^{*}\Vert\rightarrow \infty$ and $V(x) = 0$ if and only if $x=x^*$.

In the following, we consider $x_0\in \mathbb{R}^{n}\setminus \{x^{*}\}$. It follows from \eqref{ie:contr}, \eqref{ie:lowerbound} and \eqref{eq:nds} that
\begin{align*}
\frac{dV(x)}{dt} &= (x - x^*)^\top \frac{dx}{dt} = -\left\langle x(t) - x^*, \gamma \rho(x)A^\top r(x)\right\rangle\\
&= -\left\langle x - x^*, \frac{\gamma\rho_1 A^\top r(x)}{\|r(x)\|^{1-\lambda_1}} + \frac{\gamma\rho_2 A^\top r(x)}{\|r(x)\|^{1-\lambda_2}}\right\rangle\\
&= - \frac{\gamma\rho_1}{\|r(x)\|^{1-\lambda_1}}\left\langle x - x^*, A^\top r(x)\right\rangle
- \frac{\gamma\rho_2}{\|r(x)\|^{1-\lambda_2}}\left\langle x - x^*,  A^\top r(x)\right\rangle\\
&\le  - \frac{\gamma\rho_1}{2\|r(x)\|^{1-\lambda_1}}\|r(x)\|^2
- \frac{\gamma\rho_2}{2\|r(x)\|^{1-\lambda_2}}\|r(x)\|^2\\
&\le  - \frac{\gamma\rho_1}{2\|r(x)\|^{1-\lambda_1}}\frac{\mu^2}{(L_1+L_2)^2}\|x-x^*\|^2
- \frac{\gamma\rho_2}{2\|r(x)\|^{1-\lambda_2}}\frac{\mu^2}{(L_1+L_2)^2}\|x-x^*\|^2\\
&\le  -\frac{\gamma\rho_1\mu^2}{2(L_1+L_2)^{3-\lambda_1}}
\|x-x^*\|^{\lambda_1+1}
- \frac{\gamma\rho_2\mu^{1+\lambda_2}}{2(L_1+L_2)^{1+\lambda_2}}\|x-x^*\|^{\lambda_2+1}\\
&= -\frac{2^{\frac{\lambda_1-1}{2}}\gamma\rho_1\mu^2}{(L_1+L_2)^{3-\lambda_1}}
\left(\frac{1}{2}\|x-x^*\|^2\right)^{\frac{\lambda_1+1}{2}}
- \frac{2^{\frac{\lambda_2-1}{2}}\gamma\rho_2\mu^{1+\lambda_2}}{(L_1+L_2)^{1+\lambda_2}}\left(\frac{1}{2}\|x-x^*\|^2\right)^{\frac{\lambda_2+1}{2}}\\
&= -c_1 V(x)^{\kappa_1} - c_2 V(x)^{\kappa_2},
\end{align*}
where
\begin{equation}\label{eq:cc}
c_1 = \frac{2^{\frac{\lambda_1-1}{2}}\gamma\rho_1\mu^2}{(L_1+L_2)^{3-\lambda_1}}>0, \quad
c_2 = \frac{2^{\frac{\lambda_2-1}{2}}\gamma\rho_2\mu^{1+\lambda_2}}{(L_1+L_2)^{1+\lambda_2}}>0,
\end{equation}
and
\begin{equation}\label{eq:kk}
\kappa_1 = \frac{\lambda_1+1}{2}\in (0.5,1),\quad \kappa_2 = \frac{\lambda_2+1}{2}\in (1,+\infty).
\end{equation}
Then the results follow from Lemma~\ref{lem:ft}.
\end{proof}

\begin{rem}
It follows from \cite{zuo2018} that the smaller settling-time often implies the faster convergent rate of a given fixed-time dynamical system. Hence, it concludes from \eqref{ie:st}, \eqref{eq:cc} and \eqref{eq:kk} that the larger $\gamma$, $\rho_1$ or $\rho_2$ is, the smaller the upper bound $T_{\max}$ of the settling-time or the faster convergence rate of the dynamical system \eqref{eq:nds} is. In addition, $T_{\max}$ is also dependent on the parameters $\lambda_1$ and $\lambda_2$. See the next section for more details.
\end{rem}

\section{Numerical simulations}\label{sec:numerical}
In this section, we will present one example to illustrate the effectiveness of the proposed method. All experiments are implemented in MATLAB R2018b with a machine precision $2.22\times 10^{-16}$ on a PC Windows 10 operating system with an Intel i7-9700 CPU and 8GB RAM. The ODE solver used is ``ode23". Concretely, the MATLAB expression
$$
[t,y] = \text{ode23}(odefun,tspan,y_0)
$$
is used, which integrates the system of differential equations from $t_0$ to $t_f$ with $tspan = [t_0,t_f]$. Here, ``odefun'' is a function handle.

\begin{exam}[\!\!\cite{guwl2019}]\label{ex:sorvsdrs} Consider the AVEs \eqref{eq:ave} with
$$ A = \textbf{tridiag}(-1,8,-1)\in \mathbb{R}^{n\times n}\quad\text{and}\quad b=Ax^*-|x^*|,$$
where $x^*=(-1,1,-1,1,\cdots,-1,1)^\top\in \mathbb{R}^n.$

In this example, we have $\|A^{-1}\|<1$ and thus the AVEs \eqref{eq:ave} has a unique solution for any $b\in\mathbb{R}^n$. Equivalently, the dynamical model \eqref{eq:nds} has a unique equilibrium point and it is globally fixed-time stable.

Since we are interested in the fixed-time stable dynamical systems, in the first stage, we only compare our method (denoted by `LYYHC') with the method (denoted by `JLHH') proposed in \cite{ju2022}. For both methods, we set $\rho_1=\rho_2=100$, $\lambda_1=0.5$, $\lambda_2 = 1.5$. For LYYHC, $\gamma = 6$ and $x^*=(0,0,\cdots,0)^\top\in \mathbb{R}^n.$ For JLHH, $\theta = 0.5*\lambda_{\min}(M)$, $\eta$ is selected such that $\eta L^2<2\theta$ and $y_0 = (M-I)x_0-b$. When $n=20$, we have $\eta = 1.2,\theta= 1.2228,T_{\max}^{JLHH} = 0.8015, T_{\max}^{LYYHC} = 0.7355$ and numerical simulations are shown in Figure~\ref{fig:gj} and Figure~\ref{fig:err}. When $n=2000$, we have $\eta = 1.2,\theta= 1.2222,T_{\max}^{JLHH} = 0.8625, T_{\max}^{LYYHC} = 0.7467$ and numerical simulations are shown in Figure~\ref{fig:err}. Under the setting of these parameters, numerical results show that LYYHC is superior to JLHH.

In the second stage, we show the influence of the tunable parameters for LYYHC. Firstly, we fix $\gamma =6, \lambda_1 = 0.5, \lambda_2 = 1.5, n=20$. Secondly, we freeze $\gamma=1,\rho_1=\rho_2=2,\lambda_1 = 0.5, n=10$. Thirdly, we fix $\gamma=5,\rho_1=\rho_2=5,\lambda_2 = 1.01, n=10$. Finally, we set $\rho_1 =\rho_2=100, \lambda_1 = 0.5, \lambda_2 = 1.5, n=20$ and alter $\gamma$. Numerical results are shown in Figure~\ref{fig:vp} and Table~\ref{table:QBD3}, for this example, from which we can observe that the settling-time of LYYHC is decreasing with increased value of parameters $\gamma, \rho_1,\rho_2$. However, the influence of the parameters $\lambda_1$ and $\lambda_2$ on $T_{\max}$ or the convergence rate of \eqref{eq:nds} is more complicated.

\begin{figure}[t]
{\centering
\begin{tabular}{ccc}
\hspace{-0.3 cm}
\resizebox*{0.50\textwidth}{0.26\textheight}{\includegraphics{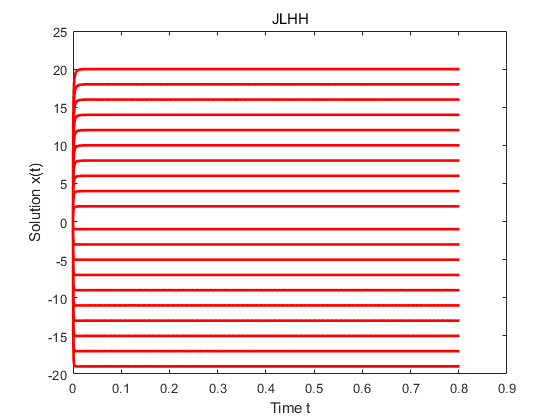}}
& & \hspace{-0.9 cm}
\resizebox*{0.50\textwidth}{0.26\textheight}{\includegraphics{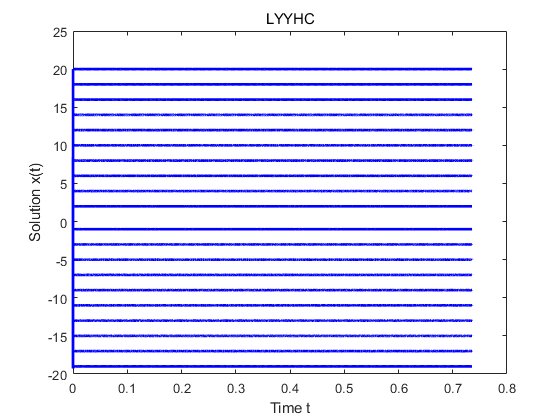}} \vspace{2ex}
\end{tabular}\par
}\vspace{-0.15 cm}
\caption{Phase diagrams for Example~\ref{ex:sorvsdrs} with $n=20$.}
\label{fig:gj}
\end{figure}

\begin{figure}[t]
{\centering
\begin{tabular}{ccc}
\hspace{-0.3 cm}
\resizebox*{0.50\textwidth}{0.26\textheight}{\includegraphics{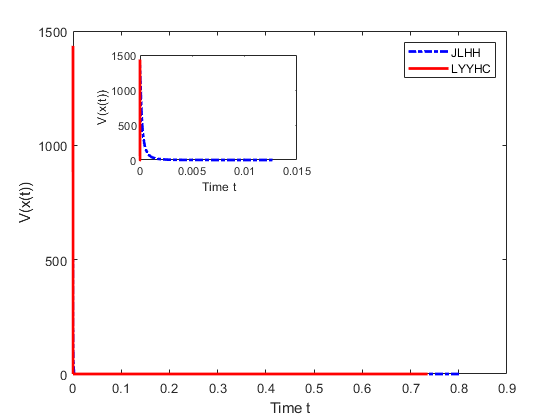}}
& & \hspace{-0.9 cm}
\resizebox*{0.50\textwidth}{0.26\textheight}{\includegraphics{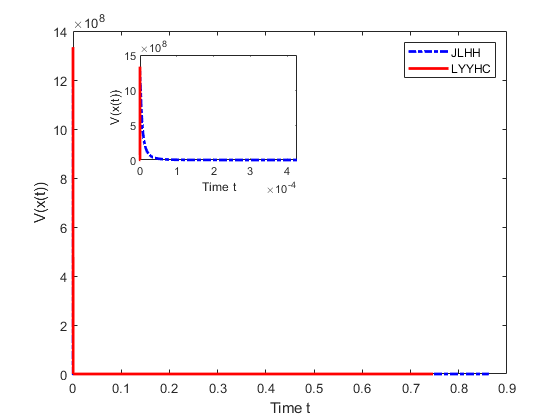}} \vspace{2ex}
\end{tabular}\par
}\vspace{-0.15 cm}
\caption{The energy $V(x(t))$ for Example~\ref{ex:sorvsdrs} with $n=20$ (left) and $n=2000$ (right). The energy of the first $30$ points in the trajectory is shown in the inner figure, from which we clearly see that LYYHC converges faster than LJHH.}
\label{fig:err}
\end{figure}

\begin{figure}[t]
{\centering
\begin{tabular}{ccc}
\hspace{-0.3 cm}
\resizebox*{0.50\textwidth}{0.26\textheight}{\includegraphics{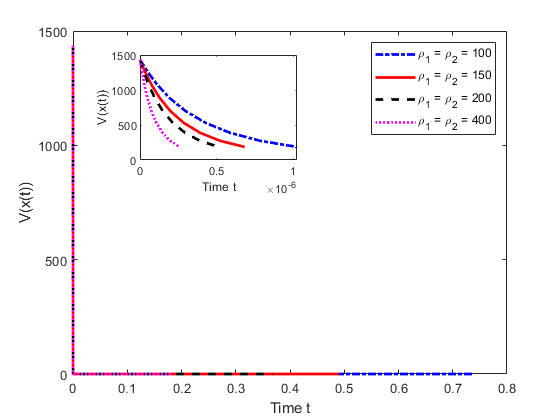}}
& & \hspace{-0.9 cm}
\resizebox*{0.50\textwidth}{0.26\textheight}{\includegraphics{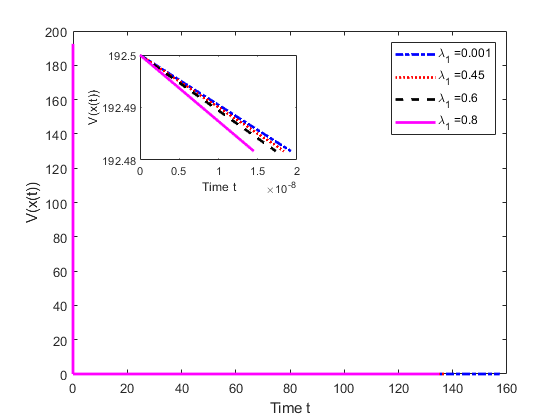}} \vspace{2ex}\\
\hspace{-0.3 cm}
\resizebox*{0.50\textwidth}{0.26\textheight}{\includegraphics{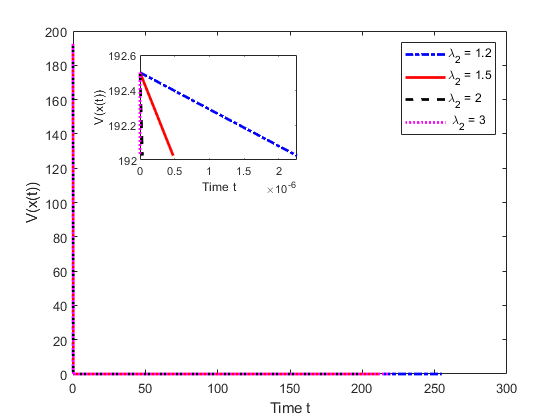}}
& & \hspace{-0.9 cm}
\resizebox*{0.50\textwidth}{0.26\textheight}{\includegraphics{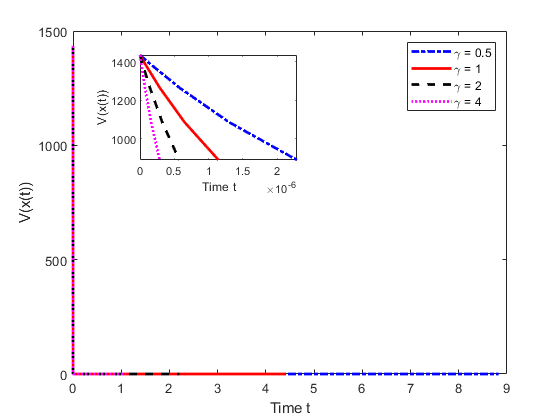}} \vspace{2ex}
\end{tabular}\par
}\vspace{-0.15 cm}
\caption{The energy $V(x(t))$ for Example~\ref{ex:sorvsdrs} with different parameters.}
\label{fig:vp}
\end{figure}

%%%%%%%%%%%%%%%%%%%%%%%%%%%%%%%%%%%%%%%%%%%
\setlength{\tabcolsep}{6.0pt}
\begin{table}[!h]\small
\centering
\caption{$T_{\max}$ for Example \ref{ex:sorvsdrs}.}\label{table:QBD3}
\begin{tabular}{|c|cccc|}\hline
$\lambda_1$  & $0.001$ & $0.45$ & $0.6$ & $0.8$\\
$T_{\max}$ & $157.3289$ & $137.0701$ & $134.7652$ & $135.0736$\\ \cline{1-5}
$\lambda_2$  & $1.2$ & $1.5$ & $2$ & $3$\\
$T_{\max}$ & $254.8676$ & $212.1412$ & $203.5690$ & $219.5776$\\ \cline{1-5}
$\gamma$  & $0.5$ & $1$ & $2$ & $4$\\
$T_{\max}$ & $8.8264$ & $4.4132$ & $2.2066$ & $1.1033$\\ \cline{1-5}
$\rho_1 (\text{or}~\rho_2)$  & $100$ & $150$ & $200$ & $400$\\
$T_{\max}$ & $0.7355$ & $0.4904$ & $0.3678$ & $0.1839$\\ \cline{1-5}
\end{tabular}
\end{table}

\end{exam}

\section{Conclusion}\label{sec:conclusion}
In this paper, a fixed-time dynamical system is proposed to solve the AVEs~\eqref{eq:ave}. Theoretical results show that the presented model is globally convergent and has a conservative settling-time. Numerical comparison with the method proposed in \cite{ju2022} shows that our method is preferred, at least in the sense that our method is inverse-free.

\end{document}